\documentclass{amsart}
\usepackage{amsfonts}

\newtheorem{theorem}{Theorem}[section]

\newtheorem{proposition}[theorem]{Proposition}
\newtheorem{remark}[theorem]{Remark}
\numberwithin{theorem}{section}

\input{tcilatex}

\begin{document}
\title[Covariant version of the Stinespring theorem for Hilbert $C^{\ast }$%
-modules]{Covariant version of the Stinespring type theorem for Hilbert $%
C^{\ast }$-modules}
\author{Maria Joi\c{t}a}
\address{Department of Mathematics \\
University of Bucharest\\
Bd. Regina Elisabeta nr. 4-12\\
Bucharest, Romania\\
Phone: + 4 - 031- 8066770}
\email{mjoita@fmi.unibuc.ro}
\urladdr{http://sites.google.com/a/g.unibuc.ro/maria-joita/}

\begin{abstract}
We prove a covariant version of the Stinespring theorem for Hilbert $C^{\ast
}$-modules.
\end{abstract}

\maketitle

\section{\protect\bigskip Introduction}

A completely positive linear map from a $C^{\ast }$-algebra $A$ to another $%
C^{\ast }$-algebra $B$ is a map $\varphi :A\rightarrow B$ with the property
that $\left[ \varphi \left( a_{ij}\right) \right] _{i.j=1}^{n}$ is a
positive element in the $C^{\ast }$-algebra $M_{n}(B)$ of all $n\times n$
matrices with elements in $B$ for all positive matrices $\left[ a_{ij}\right]
_{i.j=1}^{n}$ in $M_{n}(A)$ and for all positive integers $n$.$\ $The study
of completely positive maps is motivated by the applications of the theory
of completely positive maps to quantum information theory (operator valued
completely positive maps on $C^{\ast }$-algebras are used as mathematical
model for quantum operations) and quantum probability.

Sitinespring \cite{9} shown that a completely positive map $\varphi
:A\rightarrow L(H)$ is of the form $\varphi \left( \cdot \right) =V^{\ast
}\pi \left( \cdot \right) V$ where $\pi $ is a $\ast $-representation of $A$
on a Hilbert space $K$ and $V\ $is a bounded linear operator from $H\ $to $K.
$

Hilbert $C^{\ast }$-modules are generalizations of Hilbert spaces and $%
C^{\ast }$-algebras. In \cite{3} it is proved a version of the Stinespring
theorem for completely positive map on Hilbert $C^{\ast }$-modules. In this
paper, we will prove a version of the covariant Stinespring theorem for
Hilbert $C^{\ast }$-modules.

A Hilbert $C^{\ast }$-module $X$ over a $C^{\ast }$-algebra $A$ (or a
Hilbert $A$-module) is a linear space that is also a right $A$-module,
equipped with an $A$-valued inner product $\left\langle \cdot ,\cdot
\right\rangle $ that is $\mathbb{C}$- and $A$-linear in the second variable
and conjugate linear in the first variable such that $X$ is complete with
the norm $\left\Vert x\right\Vert =\left\Vert \left\langle x,x\right\rangle
\right\Vert ^{\frac{1}{2}}.$ $X$ is full if the closed bilateral $\ast $%
-sided ideal $\left\langle X,X\right\rangle $ of $A$ generated by $%
\{\left\langle x,y\right\rangle ;x,y\in X\}$ coincides with $A.$

\textit{A} \textit{representation }of $X$ on the Hilbert spaces $H$ and $K$
is a map $\pi _{X}:X\rightarrow L(H,K)$ with the property that there is a $%
\ast $-representation $\pi _{A}$ of $A$ on the Hilbert space $H$ such that%
\begin{equation*}
\left\langle \pi _{X}(x),\pi _{X}(y)\right\rangle =\pi _{A}\left(
\left\langle x,y\right\rangle \right) 
\end{equation*}%
for all $x,y\in X$. If $X$ is full, then the $\ast $-representation $\pi _{A}
$ associated to $\pi _{X}$ is unique. A representation $\pi
_{X}:X\rightarrow L(H,K)$ of $X$ is \textit{nondegenerate} if $\left[ \pi
_{X}(X)H\right] =K$ and $\left[ \pi _{X}(X)^{\ast }K\right] =H$ (here, $[Y]$
denotes the closed subspace of a Hilbert space $Z$ generated by the subset $%
Y\subseteq Z$). Two representations $\pi _{X}:X\rightarrow L(H,K)$ and $\pi
_{X}^{^{\prime }}:X\rightarrow L(H^{^{\prime }},K^{^{\prime }})$ are \textit{%
unitarily equivalent }if there are two unitary operators $U_{1}\in
L(H,H^{^{\prime }})$ and $U_{2}\in L(K,K^{^{\prime }})$ such that $U_{2}\pi
_{X}(x)=\pi _{X}^{^{\prime }}(x)U_{1}$ for all $x$ in $X$ \cite{1}.

A map $\Phi :X\rightarrow L(H,K)$ is called \textit{a completely positive
map on }$X$ if there is a completely positive linear map $\varphi
:A\rightarrow L(H)$ such that%
\begin{equation*}
\left\langle \Phi \left( x\right) ,\Phi \left( y\right) \right\rangle
=\varphi \left( \left\langle x,y\right\rangle \right)
\end{equation*}%
for all $x$ and $y$ in $X$. If $X$ is full, then the completely positive map 
$\varphi $ associated to $\Phi $ is unique. If $\Phi :X\rightarrow L(H,K)$
is a completely positive map on $X$, then $\Phi $ is linear and continuous.

B.V.R. Bhat, G. Ramesh, and K. Sumesh \cite{3} provided a Stinespring
construction associated to a completely positive map $\Phi $ on a Hilbert $%
C^{\ast }$-module $X$ in terms of the Stinespring construction associated to
the underlying completely positive map $\varphi .$ In Section 2 we present a
Stinespring construction associated to a completely positive map on a full
Hilbert $C^{\ast }$-module, the construction is similar to the construction
given in \cite[Theorem 2.1]{3} but it is not given in terms of the
underlying completely positive map.

A\textit{\ morphism of Hilbert }$C^{\ast }$\textit{-modules} \cite{2} or a
generalized isometry \cite{10} is a map $\Psi :X\rightarrow Y$ from a
Hilbert $A$-module $X$ to a Hilbert $B$-module $Y$ with the property that
there is a $C^{\ast }$-morphism $\psi :A\rightarrow B$ such that 
\begin{equation*}
\left\langle \Psi \left( x\right) ,\Psi \left( y\right) \right\rangle =\psi
\left( \left\langle x,y\right\rangle \right) 
\end{equation*}%
for all $x$ and $y$ in $X$. If $X$ is full, then the underlying $C^{\ast }$%
-morphism of $\Psi $ is unique, in fact $\Psi $ is a ternary morphism \cite[%
Theorem 2.1]{10}.  A map $\Psi :X\rightarrow Y$ is \textit{an isomorphism of
Hilbert }$C^{\ast }$\textit{-modules} if it is invertible, $\Psi $ and $\Psi
^{-1}$ are morphisms of Hilbert $C^{\ast }$-modules.

Suppose that $G$ is a locally compact group, $\Delta $ is the modular
function of $G$ with respect to left invariant Haar measure $ds$. \textit{A
continuous action} of $G$ on a full Hilbert $A$-module $X$ is a group
morphism $t\mapsto \eta _{t}$ from $G$ to Aut$(X)$, the group of all
isomorphisms of Hilbert $C^{\ast }$-modules from $X$ to $X$, such that the
map $(t,x)\mapsto \eta _{t}\left( x\right) $ from $G\times X$ to $X$ is
continuous. The triple $\left( G,\eta ,X\right) $ will be called a dynamical
system on Hilbert $C^{\ast }$-modules. Any $C^{\ast }$-dynamical system $%
\left( G,\alpha ,A\right) $ can be regarded as a dynamical system on Hilbert 
$C^{\ast }$-modules.

Let $t\mapsto u_{t}$ and $t\mapsto u_{t}^{\prime }$ be two unitary $\ast $%
-representations of $G$ on the Hilbert spaces $H$ and $K$. A completely
positive map $\Phi :X\rightarrow L(H,K)$ is $\left( u^{\prime },u\right) $%
\textit{-covariant} with respect to $(G,\eta ,X)$ if 
\begin{equation*}
\Phi \left( \eta _{t}\left( x\right) \right) =u_{t}^{\prime }\Phi \left(
x\right) u_{t}^{\ast }
\end{equation*}%
for all $x\in X$ and for all $t\in G$. Clearly, if $\Phi :A\rightarrow L(H)$
is a completely positive map $u$-covariant with respect to the $C^{\ast }$%
-dynamical system $(G,\alpha ,A)$, then it is  $\left( u,u\right) $%
-covariant with respect to the dynamical system on Hilbert $C^{\ast }$%
-modules, $(G,\alpha ,A).$

In Section 3, we provide a covariant version of the Stinespring theorem, and
in Section 4, we show that any covariant completely positive map $\Phi $
with respect to $\left( G,\eta ,X\right) $ induces a completely positive map
on the crossed product $G\times _{\eta }X.$

\section{The Stinespring type theorem for Hilbert $C^{\ast }$-modules}

\begin{proposition}
Let $\pi _{X}:X\rightarrow L(H,K)$ be a representation of $X$, $V\in L(H)$
and $W\in L(K)$ a coisometry. Then the map $\Phi :X\rightarrow L(H,K)$
defined by 
\begin{equation*}
\Phi \left( x\right) =W^{\ast }\pi _{X}\left( x\right) V
\end{equation*}%
for all $x\in X$ is a completely positive map.
\end{proposition}

\begin{proof}
Indeed, we have 
\begin{eqnarray*}
\left\langle \Phi \left( x\right) ,\Phi \left( y\right) \right\rangle 
&=&\left\langle W^{\ast }\pi _{X}\left( x\right) V,W^{\ast }\pi _{X}\left(
y\right) V\right\rangle  \\
&=&\left\langle \pi _{X}\left( x\right) V,\pi _{X}\left( y\right)
V\right\rangle =V^{\ast }\pi _{A}\left( \left\langle x,y\right\rangle
\right) V
\end{eqnarray*}%
for all $x,y\in X$, and since the map $\varphi :A\rightarrow L(H)$ defined
by 
\begin{equation*}
\varphi \left( a\right) =V^{\ast }\pi _{A}\left( a\right) V
\end{equation*}%
is completely positive, $\Phi $ is completely positive.
\end{proof}

We show that an operator valued completely positive linear map $\Phi $ on a
full Hilbert $C^{\ast }$-module $X$ is of the form $\Phi \left( \cdot
\right) =W^{\ast }\pi _{X}\left( \cdot \right) V$, where $\pi _{X}$ is a
representation of $X$, $W\ $is a coisometry and $V$ is a bounded linear map.
Moreover, under some conditions this writing is unique up to unitary
equivalence.

\begin{theorem}
Let $X$ be a full Hilbert $C^{\ast }$-module over a $C^{\ast }$-algebra $A$, 
$H$ and $K$ two Hilbert spaces and $\Phi :X\rightarrow L(H,K)$ a completely
positive map. Then:

\begin{enumerate}
\item There are two Hilbert spaces $H_{\Phi }$ and $K_{\Phi }$, a
representation $\pi _{\Phi }:X\rightarrow L(H_{\Phi },K_{\Phi })$ of $X$, a
bounded linear operator $V_{\Phi }:H\rightarrow H_{\Phi }$ and a coisometry $%
W_{\Phi }:K\rightarrow K_{\Phi }$ such that:

\begin{enumerate}
\item $\Phi \left( x\right) =W_{\Phi }^{\ast }$ $\pi _{\Phi }\left( x\right)
V_{\Phi }$ for all $x\in X;$

\item $\left[ \pi _{\Phi }\left( X\right) V_{\Phi }H\right] =K_{\Phi };$

\item $\left[ \pi _{\Phi }\left( X\right) ^{\ast }W_{\Phi }K\right] =H_{\Phi
}.$
\end{enumerate}

\item If $H^{^{\prime }}$and $K^{\prime }$ are two Hilbert spaces, $\pi
_{X}:X\rightarrow L(H^{\prime },K^{\prime })$ a representation of $X,$ $%
V^{\prime }$ an element in $L(H,H^{\prime })$ and $W^{\prime }:K\rightarrow
K^{\prime }$ a coisometry that verify the following relations:

\begin{enumerate}
\item $\Phi \left( x\right) =W^{\prime \ast }$ $\pi _{X}\left( x\right)
V^{\prime }$ for all $x\in X;$

\item $\left[ \pi _{X}\left( X\right) V^{\prime }H\right] =K^{\prime };$

\item $\left[ \pi _{X}\left( X\right) ^{\ast }W^{\prime }K\right] =H^{\prime
},$
\end{enumerate}
\end{enumerate}

then there are two unitary operators $U_{1}\in L(H_{\Phi },H^{\prime })$ and 
$U_{2}\in L(K_{\Phi },K^{\prime })$ such that: $U_{2}\pi _{\Phi }\left(
x\right) =\pi _{X}\left( x\right) U_{1}\ $for all $x\in X,$ $V^{\prime
}=U_{1}V_{\Phi }$ and $W^{\prime }=U_{2}W_{\Phi }.$
\end{theorem}

\begin{proof}
$\left( 1\right) $ Let $\varphi $ be the completely positive linear map
associated to $\Phi $ and let $(\pi _{\varphi },H_{\varphi },V_{\varphi })$
be the Stinespring construction associated to $\varphi $ \cite[Theorem 5.6
(1)]{7}. Let $H_{\Phi }=H_{\varphi }$, $V_{\Phi }=V_{\varphi },K_{\Phi }=%
\left[ \Phi (X)H\right] $ and $W_{\Phi }$ the projection of $K$ on $K_{\Phi }
$. Exactly as in the proof of Theorem 2.1 \cite{3} it is shown that the map $%
\pi _{\Phi }:X\rightarrow L(H_{\Phi },K_{\Phi })$ defined by $\pi _{\Phi
}\left( x\right) \left( \tsum\limits_{i=1}^{n}\pi _{\varphi }\left(
a_{i}\right) V_{\Phi }h_{i}\right) =\tsum\limits_{i=1}^{n}\Phi \left(
xa_{i}\right) h_{i}$ is a representation of $X$ that verifies the relations $%
(a)$ and $(b)$. From%
\begin{eqnarray*}
\left[ \pi _{\Phi }\left( X\right) ^{\ast }W_{\Phi }K\right]  &=&\left[ \pi
_{\Phi }\left( X\right) ^{\ast }K_{\Phi }\right] =\left[ \pi _{\Phi }\left(
X\right) ^{\ast }\pi _{\Phi }\left( X\right) V_{\Phi }H\right]  \\
&=&\left[ \pi _{\varphi }\left( \left\langle X,X\right\rangle \right)
V_{\Phi }H\right] =\left[ \pi _{\varphi }\left( A\right) V_{\Phi }H\right]
=H_{\Phi }
\end{eqnarray*}%
we deduce that the relation $(c)$ is verified too.

$\left( 2\right) $ If $\pi _{A}$:$A\rightarrow L(H^{\prime })$ is the $\ast $%
-representation associated to $\pi _{X}$, then: 
\begin{eqnarray*}
\varphi \left( \left\langle x,y\right\rangle \right)  &=&\left\langle \Phi
\left( x\right) ,\Phi \left( y\right) \right\rangle =\left( W^{\prime \ast
}\pi _{X}\left( x\right) V^{\prime }\right) ^{\ast }W^{\prime \ast }\pi
_{X}\left( y\right) V^{\prime } \\
&=&V^{\prime \ast }\pi _{X}\left( x\right) W^{\prime }W^{\prime \ast }\pi
_{X}\left( y\right) V^{\prime }=V^{\prime \ast }\pi _{A}\left( \left\langle
x,y\right\rangle \right) V^{\prime }
\end{eqnarray*}%
for all $x$ and $y$ in $X$, and%
\begin{eqnarray*}
\left[ \pi _{A}\left( \left\langle X,X\right\rangle \right) V^{\prime }H%
\right]  &=&\left[ \pi _{X}\left( X\right) ^{\ast }\pi _{X}\left( X\right)
V^{\prime }H\right] =\left[ \pi _{X}\left( X\right) ^{\ast }K^{\prime }%
\right]  \\
&=&\left[ \pi _{X}\left( X\right) ^{\ast }W^{\prime }K\right] =H^{\prime }.
\end{eqnarray*}%
Therefore, $\left( \pi _{A},H^{\prime },V^{\prime }\right) $ is unitarily
equivalent to the Stinespring construction associated to $\varphi $ \cite[%
Theorem 5.6 (2)]{7}, and so there is a unitary operator $U_{1}\in L(H_{\Phi
},H^{\prime })$ such that $\pi _{A}(a)=U_{1}{}\pi _{\varphi }\left( a\right)
U_{1}^{\ast }$ and $V^{\prime }=U_{1}V_{\Phi }$. As in the proof of Theorem
2.4 \cite{3} we show that the there is a unitary operator $U_{2}:K_{\Phi }$ $%
\rightarrow K^{\prime }$ such that 
\begin{equation*}
U_{2}\left( \tsum\limits_{i=1}^{n}\pi _{\Phi }\left( x_{i}\right) V_{\Phi
}h_{i}\right) =\tsum\limits_{i=1}^{n}\pi _{X}\left( x_{i}\right) V^{\prime
}h_{i}
\end{equation*}%
and moreover, 
\begin{equation*}
U_{2}\pi _{\Phi }\left( x\right) =\pi _{X}\left( x\right) U_{1}\text{ and }%
W^{\prime }=U_{2}W_{\Phi }.
\end{equation*}
\end{proof}

\section{The covariant version of the Stinespring construction}

Let $(G,\eta ,X)$ be a dynamical system on Hilbert $C^{\ast }$-modules. 
\textit{A covariant representation }of $(G,\eta ,X)$ is a quadruple $\left(
\pi _{X},v,w,H,K\right) $ consists of two Hilbert spaces $H$ and $K$, a
representation $\pi _{X}:X\rightarrow L(H,K)$ of $X$, a unitary $\ast $%
-representation of $G$ on $H$, $t\mapsto v_{t}$, and a unitary $\ast $%
-representation of $G$ on $K$, $t\mapsto w_{t}$ such that 
\begin{equation*}
\pi _{X}\left( \eta _{t}\left( x\right) \right) =w_{t}\pi _{X}\left(
x\right) v_{t}^{\ast }
\end{equation*}%
for all $x\in X$ and for all $t\in G$. We say that the covariant
representation $\left( \pi _{X},v,w,H,K\right) $ is nondegenerate if the
representation $\pi _{X}$ is nondegenerate. Clearly, any covariant
representation of a $C^{\ast }$-dynamical system $\left( G,\alpha ,A\right) $
is a covariant representation of $\left( G,\alpha ,A\right) $ regarded as
dynamical system on Hilbert $C^{\ast }$-modules.

Any continuous action $t\mapsto \eta _{t}$ of $G$ on $X$ induces a unique
continuous action $t\mapsto \alpha _{t}^{\eta }$ of $G$ on $A$ such that $%
\alpha _{t}^{\eta }\left( \left\langle x,y\right\rangle \right)
=\left\langle \eta _{t}\left( x\right) ,\eta _{t}\left( x\right)
\right\rangle $ for all $x,y\in X$ and for all $t\in G$ \cite{4}.

\begin{remark}
A (nondegenerate) \textit{covariant representation }$\left( \pi
_{X},v,w,H,K\right) $ of $(G,\eta ,X)$ induces a (nondegenerate)
representation of $(G,\alpha ^{\eta },A)$. Indeed, if $\pi _{A}$ is the $%
\ast $-representation associated to $\pi _{X}$, then%
\begin{eqnarray*}
\pi _{A}\left( \alpha _{t}^{\eta }\left( \left\langle x,y\right\rangle
\right) \right) &=&\pi _{A}\left( \left\langle \eta _{t}\left( x\right)
,\eta _{t}\left( y\right) \right\rangle \right) =\left\langle \pi _{X}\left(
\eta _{t}\left( x\right) \right) ,\pi _{X}\left( \eta _{t}\left( y\right)
\right) \right\rangle \\
&=&\left\langle w_{t}\pi _{X}\left( x\right) v_{t}^{\ast },w_{t}\pi
_{X}\left( y\right) v_{t}^{\ast }\right\rangle =v_{t}\pi _{A}\left(
\left\langle x,y\right\rangle \right) v_{t}^{\ast }
\end{eqnarray*}%
for all $x,y\in X$ and for all $t\in G$. Therefore $\left( \pi
_{A},v,H\right) $ is a covariant representation of $(G,\alpha ^{\eta },A)$.
\end{remark}

Let $t\mapsto u_{t}$ and $t\mapsto $ $u_{t}^{\prime }$ be two unitary $\ast $%
-representations of $G$ on the Hilbert spaces $H$ and $K$.

\begin{remark}
If $\Phi :X\rightarrow L(H,K)$ is a completely positive map, $\left(
u^{\prime },u\right) $-covariant with respect to $(G,\eta ,X)$, then the
completely positive map $\varphi $ associated to $\Phi $ is $u$-covariant
with respect to $(G,\alpha ^{\eta },A)$.

Indeed, we have 
\begin{eqnarray*}
\varphi \left( \alpha _{t}^{\eta }\left( \left\langle x,y\right\rangle
\right) \right) &=&\varphi \left( \left\langle \eta _{t}\left( x\right)
,\eta _{t}\left( y\right) \right\rangle \right) =\left\langle \Phi \left(
\eta _{t}\left( x\right) \right) ,\Phi \left( \eta _{t}\left( y\right)
\right) \right\rangle \\
&=&\left\langle u_{t}^{\prime }\Phi \left( x\right) u_{t}^{\ast
},u_{t}^{\prime }\Phi \left( y\right) u_{t}^{\ast }\right\rangle
=u_{t}\varphi \left( \left\langle x,y\right\rangle \right) u_{t}^{\ast }
\end{eqnarray*}%
for all $x,y\in X$ and for all $t\in G$.
\end{remark}

\begin{proposition}
Let $\left( \pi _{X},v,w,H,K\right) $ be a covariant representation of $%
(G,\eta ,X)$, $V\in L(H)$, $W\in L(K)$ a coisometry, $t\mapsto u_{t}$ and $%
t\mapsto u_{t}^{\prime }$ two unitary $\ast $-representations of $G$ on $H$
respectively $K$ such that $v_{t}V=Vu_{t}$ and $w_{t}W=Wu_{t}^{\prime }$ \
for all $t\in G$. Then the map $\Phi :X\rightarrow L(H,K)$ defined by 
\begin{equation*}
\Phi \left( x\right) =W^{\ast }\pi _{X}\left( x\right) V
\end{equation*}%
for all $x\in X$ is a completely positive map, $\left( u^{\prime },u\right) $%
-covariant with respect to $(G,\eta ,X)$.
\end{proposition}

\begin{proof}
By Proposition 2.1, the map $\Phi $ is completely positive. From 
\begin{equation*}
\Phi \left( \eta _{t}\left( x\right) \right) =W^{\ast }\pi _{X}\left( \eta
_{t}\left( x\right) \right) V=W^{\ast }w_{t}\pi _{X}\left( x\right)
v_{t}^{\ast }V=u_{t}^{\prime }W^{\ast }\pi _{X}\left( x\right) Vu_{t}^{\ast
}=u_{t}^{\prime }\Phi \left( x\right) u_{t}^{\ast }
\end{equation*}%
for all $x\in X$ and for all $t\in G$, we deduce that the completely
positive map $\Phi $ is $\left( u^{\prime },u\right) $-covariant.
\end{proof}

We show that an operator valued $\left( u^{\prime },u\right) $-covariant
completely positive map $\Phi $ on a full Hilbert $C^{\ast }$-module $X$ is
of the form $\Phi \left( \cdot \right) =W^{\ast }\pi _{X}\left( \cdot
\right) V$, where $\left( \pi _{X},v,w,H,K\right) $ is a covariant
representation of $(G,\eta ,X)$, $W\ $is a coisometry such that $%
w_{t}W=Wu_{t}^{\prime }$ for all $t\in G$ and $V$ is a bounded linear map
such that $v_{t}V=Vu_{t}$ \ for all $t\in G$. Moreover, under some
conditions this writing is unique up to unitary equivalence.

\begin{theorem}
Let $\Phi :X\rightarrow L(H,K)$ be a completely positive map, $\left(
u^{\prime },u\right) $-covariant with respect to $(G,\eta ,X)$. Then:

\begin{enumerate}
\item There are two Hilbert spaces $H_{\Phi }$ and $K_{\Phi }$, a covariant
representation $(\pi _{\Phi },v^{\Phi },$ $w^{\Phi },H_{\Phi },K_{\Phi })$
of $(G,\eta ,X)$, a linear operator $V_{\Phi }:H\rightarrow H_{\Phi }$ and a
coisometry $W_{\Phi }:K\rightarrow K_{\Phi }$ such that:

\begin{enumerate}
\item $\Phi \left( x\right) =W_{\Phi }^{\ast }$ $\pi _{\Phi }\left( x\right)
V_{\Phi }$ for all $x\in X;$

\item $v_{t}^{\Phi }V_{\Phi }=V_{\Phi }u_{t}$\ for all $t\in G;$

\item $w_{t}^{\Phi }W_{\Phi }=W_{\Phi }u_{t}^{\prime }$ for all $t\in G$

\item $\left[ \pi _{\Phi }\left( X\right) V_{\Phi }H\right] =K_{\Phi };$

\item $\left[ \pi _{\Phi }\left( X\right) ^{\ast }W_{\Phi }K\right] =H_{\Phi
}.$
\end{enumerate}

\item If $H^{^{\prime }}$and $K^{\prime }$ are two Hilbert spaces, $\left(
\pi _{X},v,w,H^{\prime },K^{\prime }\right) $ a covariant representation of $%
(G,\eta ,X)$, $V^{\prime }$ an element in $L(H,H^{\prime })$ and $W^{\prime
}:K\rightarrow K^{\prime }$ a coisometry which verify the following
relations:

\begin{enumerate}
\item $\Phi \left( x\right) =W^{\prime \ast }$ $\pi _{X}\left( x\right)
V^{\prime }$ for all $x\in X;$

\item $v_{t}V^{\prime }=V^{\prime }u_{t}\ $for all $t\in G;$

\item $w_{t}W^{\prime }=W^{\prime }u_{t}^{\prime }\ $ for all $t\in G;$

\item $\left[ \pi _{X}\left( X\right) V^{\prime }H\right] =K^{\prime };$

\item $\left[ \pi _{X}\left( X\right) ^{\ast }W^{\prime }K\right] =H^{\prime
},$
\end{enumerate}

then there are two unitary operators $U_{1}\in L(H_{\Phi },H^{\prime })$ and 
$U_{2}\in L(K_{\Phi },K^{\prime })$ such that: $U_{2}\pi _{\Phi }\left(
x\right) =\pi _{X}\left( x\right) U_{1},$ $v_{t}U_{1}=U_{1}v_{t}^{\Phi },$ $%
w_{t}U_{2}=U_{2}w_{t}^{\Phi },$ $V_{t}^{\prime }=U_{1}V_{\Phi }$ and $%
W^{\prime }=U_{2}W_{\Phi }.$
\end{enumerate}
\end{theorem}

\begin{proof}
$(1)$ Let $\varphi $ be the completely positive map associated to $\Phi $.
Then, by Remark 3.2, $\varphi $ is $u$-covariant with respect to $(G,\alpha
^{\eta },A)$. Let $\left( \pi _{\varphi },v^{\varphi },H_{\varphi
},V_{\varphi }\right) $ be the covariant Stinespring construction associated
to $\varphi $ (see, for example, \cite{8}). Then $\left( \pi _{\varphi
},H_{\varphi },V_{\varphi }\right) $ is the Stinespring construction
associated to $\varphi $ and by Theorem 2.2, $\left( \pi _{\Phi },H_{\Phi
},K_{\Phi },V_{\Phi },W_{\Phi }\right) $, where $H_{\Phi }=H_{\varphi }$, $%
V_{\Phi }=V_{\varphi },K_{\Phi }=\left[ \Phi (X)H\right] $ , $W_{\Phi }$ is
the projection of $K$ on $K_{\Phi }$ and $\pi _{\Phi }:X\rightarrow
L(H_{\Phi },K_{\Phi })$ is defined by $\pi _{\Phi }\left( x\right) \left(
\tsum\limits_{i=1}^{n}\pi _{\varphi }\left( a_{i}\right) V_{\Phi
}h_{i}\right) $ $=\tsum\limits_{i=1}^{n}\Phi \left( xa_{i}\right) h_{i}$ is
the Sinespring construction associated to $\Phi $. Moreover, the relations $%
(a),$ $(d)$ and $(e)$ are verified.

Let $v_{t}^{\Phi }=v_{t}^{\varphi }$ for all $t\in G$. Then $t\mapsto
v_{t}^{\Phi }$ is a unitary $\ast $-representation of $G$ on $H_{\Phi }$
which verifies the relation $(b)$. Since $\Phi $ is $(u^{\prime },u)$%
-covariant, $u_{t}^{\prime }\left( \tsum\limits_{i=1}^{n}\Phi \left(
x_{i}\right) h_{i}\right) $ $=\tsum\limits_{i=1}^{n}{}\Phi \left( \eta
_{t}\left( x_{i}\right) \right) u_{t}h_{i}$ for all $t\in G$ and for all $%
x_{i}\in X,h_{i}\in H,$ $i=1,...,n,$ and so $\left[ \Phi (X)H\right] $ is
invariant under $u^{\prime }$. Then, since $W_{\Phi }$ is the projection on $%
\left[ \Phi (X)H\right] $, we have $u_{t}^{\prime }W_{\Phi }=W_{\Phi
}u_{t}^{\prime }$. Let $w_{t}^{\Phi }=u_{t}^{\prime }|_{K_{\Phi }}$ for all $%
t\in G$. Then $t\mapsto w_{t}^{\Phi }$ is a unitary $\ast $-representation
of $G$ on $K_{\Phi }$ which verifies the relation $(c)$.

To prove the assertion $(1)$ it remains to show that $\left( \pi _{\Phi
},v^{\Phi },w^{\Phi },H_{\Phi },K_{\Phi }\right) $ is a covariant
representation of $(G,\eta ,X)$. From 
\begin{eqnarray*}
\pi _{\Phi }\left( \eta _{t}\left( x\right) \right) \left(
\tsum\limits_{i=1}^{n}\pi _{\varphi }\left( a_{i}\right) V_{\varphi
}h_{i}\right)  &=&\tsum\limits_{i=1}^{n}\Phi \left( \eta _{t}\left( x\right)
a_{i}\right) h_{i}=\tsum\limits_{i=1}^{n}\Phi \left( \eta _{t}\left( x\alpha
_{t^{-1}}^{\eta }\left( a_{i}\right) \right) \right) h_{i} \\
&=&\tsum\limits_{i=1}^{n}u_{t}^{\prime }\Phi \left( x\alpha _{t^{-1}}^{\eta
}\left( a_{i}\right) \right) u_{t}^{\ast }h_{i}
\end{eqnarray*}%
and 
\begin{eqnarray*}
w_{t}^{\Phi }\pi _{X}\left( x\right) v_{t^{-1}}^{\Phi }\left(
\tsum\limits_{i=1}^{n}\pi _{\varphi }\left( a_{i}\right) V_{\varphi
}h_{i}\right)  &=&w_{t}^{\Phi }\pi _{X}\left( x\right) \left(
\tsum\limits_{i=1}^{n}v_{t^{-1}}^{\Phi }\pi _{\varphi }\left( a_{i}\right)
V_{\varphi }h_{i}\right)  \\
&=&w_{t}^{\Phi }\pi _{X}\left( x\right) \left( \tsum\limits_{i=1}^{n}\pi
_{\varphi }\left( \alpha _{t^{-1}}^{\eta }\left( a_{i}\right) \right)
v_{t^{-1}}^{\Phi }V_{\varphi }h_{i}\right)  \\
&=&w_{t}^{\Phi }\pi _{X}\left( x\right) \left( \tsum\limits_{i=1}^{n}\pi
_{\varphi }\left( \alpha _{t^{-1}}^{\eta }\left( a_{i}\right) \right)
V_{\varphi }u_{t}^{\ast }h_{i}\right)  \\
&=&\tsum\limits_{i=1}^{n}u_{t}^{\prime }\Phi \left( x\alpha _{t^{-1}}^{\eta
}\left( a_{i}\right) \right) u_{t}^{\ast }h_{i}
\end{eqnarray*}%
for all $a_{1},...,a_{n}\in A$ and for all $h_{1},...,h_{n}\in H$, we deduce
that $\pi _{\Phi }\left( \eta _{t}\left( x\right) \right) =w_{t}^{\Phi }\pi
_{X}\left( x\right) v_{t^{-1}}^{\Phi }$ for all $x\in X$ and for all $t\in G$%
. Therefore, $\left( \pi _{\Phi },v^{\Phi },w^{\Phi },H_{\Phi },K_{\Phi
}\right) $ is a covariant representation of $(G,\eta ,X)$.

$(2)$ Since $\left( \pi _{A},v,H^{\prime },V^{\prime }\right) $, where $\pi
_{A}$ is the underlying $\ast $-representation of $\pi _{X},$ is unitarily
equivalent to the covariant Stinespring construction associated to  $\varphi 
$, there is a unitary operator $U_{1}^{\prime }\in L(H_{\Phi },H^{\prime })$
such that $V^{\prime }=U_{1}^{\prime }V_{\Phi }$, $v_{t}U_{1}^{\prime
}=U_{1}^{\prime }v_{t}^{\Phi }\ $ for all $t\in G$, and $U_{1}^{\prime }\pi
_{\varphi }(a)=\pi _{A}(a)U_{1}^{\prime }\ $for all $a\in A$. On the other
hand, $\left( \pi _{\Phi },H_{\Phi },K_{\Phi },V_{\Phi },W_{\Phi }\right) $
is the Stinespring construction associated to $\Phi $, and then by Theorem
2.2 (2), there are two unitary operators $U_{1}\in L(H_{\Phi },H^{\prime })$
and $U_{2}\in L(K_{\Phi },K^{\prime })$ such that: $U_{2}\pi _{\Phi }\left(
x\right) =\pi _{X}\left( x\right) U_{1}\ $ for all $x\in X,$ $V^{\prime
}=U_{1}V_{\Phi }$ and $W^{\prime }=U_{2}W_{\Phi }$. Moreover, 
\begin{equation*}
U_{2}\left( \tsum\limits_{i=1}^{n}\pi _{\Phi }\left( x_{i}\right) V_{\Phi
}h_{i}\right) =\tsum\limits_{i=1}^{n}\pi _{X}\left( x_{i}\right) V^{\prime
}h_{i}
\end{equation*}%
for all $x_{1},...,x_{n}\in X$ and for all $h_{1},...,h_{n}\in H$, whence 
\begin{eqnarray*}
w_{t}U_{2}\left( \tsum\limits_{i=1}^{n}\pi _{\Phi }\left( x_{i}\right)
V_{\Phi }h_{i}\right)  &=&w_{t}\left( \tsum\limits_{i=1}^{n}\pi _{X}\left(
x_{i}\right) V^{\prime }h_{i}\right) =\tsum\limits_{i=1}^{n}\pi _{X}\left(
\eta _{t}\left( x_{i}\right) \right) v_{t}V^{\prime }h_{i} \\
&=&\tsum\limits_{i=1}^{n}\pi _{X}\left( \eta _{t}\left( x_{i}\right) \right)
V^{\prime }u_{t}h_{i}=U_{2}\left( \tsum\limits_{i=1}^{n}\pi _{\Phi }\left(
\eta _{t}\left( x_{i}\right) \right) V_{\Phi }u_{t}h_{i}\right)  \\
&=&U_{2}\left( \tsum\limits_{i=1}^{n}w_{t}^{\Phi }\pi _{\Phi }\left(
x_{i}\right) V_{\Phi }h_{i}\right) =U_{2}w_{t}^{\Phi }\left(
\tsum\limits_{i=1}^{n}\pi _{\Phi }\left( x_{i}\right) V_{\Phi }h_{i}\right) 
\end{eqnarray*}%
and so $w_{t}U_{2}=U_{2}w_{t}^{\Phi }$ for all $t\in G.$

From $U_{2}\pi _{\Phi }\left( x\right) =\pi _{X}\left( x\right) U_{1}$ for
all $x\in X,$ we deduce that $U_{1}\pi _{\varphi }(a)=\pi _{A}(a)U_{1}\ $for
all $a\in A$ and then 
\begin{equation*}
U_{1}^{\prime }\left( \pi _{\varphi }(a)V_{\Phi }h\right) =\pi
_{A}(a)U_{1}^{\prime }V_{\Phi }h=\pi _{A}(a)V^{\prime }h=\pi
_{A}(a)U_{1}V_{\Phi }h=U_{1}\left( \pi _{\varphi }(a)V_{\Phi }h\right)
\end{equation*}%
for all $a\in A$ and for all $h\in H$. From this relation, since $\left[ \pi
_{\varphi }(A)V_{\Phi }H\right] $ $=H$, we deduce that $U_{1}=U_{1}^{\prime
} $, and the assertion is proved.
\end{proof}

\section{Covariant completely positive maps and crossed products of Hilbert $%
C^{\ast }$-modules}

Let $(G,\eta ,X)$ be a dynamical system on Hilbert $C^{\ast }$-modules. The
linear space $C(G,X)$ of all continuous functions from $G$ to $X$ with
compact support has a structure of pre-Hilbert $G\times _{\alpha ^{\eta }}A$%
-module with the action of $G\times _{\alpha ^{\eta }}A$ on $C(G,X)$ given
by 
\begin{equation*}
\left( \widehat{x}f\right) \left( s\right) =\tint\limits_{G}\widehat{x}%
\left( t\right) \alpha _{t}^{\eta }\left( f\left( t^{-1}s\right) \right) dt
\end{equation*}%
for all $\widehat{x}\in C(G,X)$ and $f\in C(G,A)$ and the inner product
given by 
\begin{equation*}
\left\langle \widehat{x},\widehat{y}\right\rangle \left( s\right)
=\tint\limits_{G}\alpha _{t^{-1}}^{\eta }\left( \left\langle \widehat{x}(t),%
\widehat{y}\left( ts\right) \right\rangle \right) dt.
\end{equation*}%
The crossed product of $X$ by $\eta $, denoted by $G\times _{\eta }X$, is
the Hilbert $G\times _{\alpha ^{\eta }}A$-module obtained by the completion
of the pre-Hilbert $G\times _{\alpha ^{\eta }}A$-module $C(G,X)$ \cite{4,6}

Any covariant representation $\left( \pi _{X},v,w,H,K\right) $ of $\left(
G,\eta ,X\right) $ induces a representation $\left( \pi _{X}\times
v,H,K\right) $ of $G\times _{\eta }X$ such that%
\begin{equation*}
\left( \pi _{X}\times v\right) \left( \widehat{x}\right)
=\tint\limits_{G}\pi _{X}\left( \widehat{x}\left( t\right) \right) v_{t}dt
\end{equation*}%
for all $\widehat{x}\in C(G,X)$. Moreover, the underlying $\ast $%
-representation of $\pi _{X}\times v$ is the integral form of the covariant
representation $\left( \pi _{A},v,H\right) $ of $(G,\alpha ^{\eta },A)$
induced by $\left( \pi _{X},v,H,K\right) $ \cite{6}.

\begin{remark}
If $\left( \pi _{X},v,w,H,K\right) $ is a nondegenerate covariant
representation of $\left( G,\eta ,X\right) $, then its integral form $\left(
\pi _{X}\times v,H,K\right) $ is nondegenerate

Indeed, let $f\in C(G,A)$ and $x\in X$. Then $f_{x}\in C(G,X)$, where $%
f_{x}(s)=xf(s)$, 
\begin{equation*}
\left( \pi _{X}\times v\right) \left( f_{x}\right) =\tint\limits_{G}\pi
_{X}\left( xf\left( t\right) \right) v_{t}dt=\tint\limits_{G}\pi _{X}\left(
x\right) \pi _{A}\left( f\left( t\right) \right) v_{t}dt=\pi _{X}\left(
x\right) \left( \pi _{A}\times v\right) \left( f\right)
\end{equation*}%
and 
\begin{equation*}
\left( \pi _{X}\times v\right) \left( f_{x}\right) ^{\ast }=\left( \pi
_{A}\times v\right) \left( f\right) ^{\ast }\pi _{X}\left( x\right) ^{\ast }.
\end{equation*}%
From these facts and taking into account that $\left( \pi _{A}\times
v,H,K\right) $ and $\left( \pi _{X},v,H,K\right) $ are nondegenerate we
deduce that $\left[ \left( \pi _{X}\times v\right) \left( X\right) H\right]
=K$ and $\left[ \left( \pi _{X}\times v\right) \left( X\right) ^{\ast }K%
\right] =H.$
\end{remark}

\begin{proposition}
Let $\Phi :X\rightarrow L(H,K)$ be a completely positive map, $\left(
u^{\prime },u\right) $-covariant with respect to $(G,\eta ,X)$. Then there
is a completely positive map $\widehat{\Phi }:G\times _{\eta }X\rightarrow
L(H,K)$ such that 
\begin{equation*}
\widehat{\Phi }\left( \widehat{x}\right) =\tint\limits_{G}\Phi \left( 
\widehat{x}\left( t\right) \right) u_{t}dt
\end{equation*}%
for all $\widehat{x}\in C(G,X)$. Moreover, the completely positive map
associated to $\widehat{\Phi }$ is the map $\widehat{\varphi }:G\times
_{\alpha ^{\eta }}A\rightarrow L(H)$ such that 
\begin{equation*}
\widehat{\varphi }\left( f\right) =\tint\limits_{G}\varphi \left( f\left(
t\right) \right) u_{t}dt
\end{equation*}%
for all $f\in C\left( G,A\right) .$
\end{proposition}

\begin{proof}
Let $\left( \pi _{\Phi },v^{\Phi },w^{\Phi },H_{\Phi },K_{\Phi },V_{\Phi
},W_{\Phi }\right) $ be the covariant Stinespring construction associated to 
$\Phi $. Consider the map $\widehat{\Phi }:G\times _{\eta }X\rightarrow
L(H,K)$ defined by 
\begin{equation*}
\widehat{\Phi }\left( z\right) =W_{\Phi }^{\ast }\left( \pi _{\Phi }\times
v^{\Phi }\right) \left( z\right) V_{\Phi }.
\end{equation*}%
By Proposition 3.3 , $\widehat{\Phi }$ is completely positive and

\begin{eqnarray*}
\widehat{\Phi }\left( \widehat{x}\right)  &=&W_{\Phi }^{\ast }\left( \pi
_{\Phi }\times v^{\Phi }\right) \left( \widehat{x}\right) V_{\Phi
}=\tint\limits_{G}W_{\Phi }^{\ast }\pi _{\Phi }\left( \widehat{x}\left(
t\right) \right) v_{t}^{\Phi }V_{\Phi }dt \\
&=&\tint\limits_{G}W_{\Phi }^{\ast }\pi _{\Phi }\left( \widehat{x}\left(
t\right) \right) V_{\Phi }u_{t}dt=\tint\limits_{G}\Phi \left( \widehat{x}%
\left( t\right) \right) u_{t}dt
\end{eqnarray*}%
for all $\widehat{x}\in C(G,X)$. Since $\left( \pi _{\varphi }\times
v^{\varphi },H_{\Phi },V_{\Phi }\right) $ is the Stinespring construction
associated to $\widehat{\varphi }:G\times _{\alpha ^{\eta }}A\rightarrow L(H)
$ with%
\begin{equation*}
\widehat{\varphi }\left( f\right) =\tint\limits_{G}\varphi \left( f\left(
t\right) \right) u_{t}dt
\end{equation*}%
for all $f\in C\left( G,A\right) ,$ we have 
\begin{eqnarray*}
\left\langle \widehat{\Phi }\left( z_{1}\right) ,\widehat{\Phi }\left(
z_{1}\right) \right\rangle  &=&V_{\Phi }^{\ast }\left\langle \left( \pi
_{\Phi }\times v^{\Phi }\right) \left( z_{1}\right) ,\left( \pi _{\Phi
}\times v^{\Phi }\right) \left( z_{1}\right) \right\rangle V_{\Phi } \\
&=&V_{\Phi }^{\ast }\left( \pi _{\varphi }\times v^{\Phi }\right) \left(
\left\langle z_{1},z_{1}\right\rangle \right) V_{\Phi }=\widehat{\varphi }%
\left( \left\langle z_{1},z_{1}\right\rangle \right) 
\end{eqnarray*}%
for all $z_{1},z_{2}\in G\times _{\eta }X.$ Therefore, the completely
positive map associated to $\widehat{\Phi }$ is $\widehat{\varphi }$.
\end{proof}

\begin{remark}
Let $\Phi :X\rightarrow L(H,K)$ be a completely positive map, $\left(
u^{\prime },u\right) $-covariant with respect to $(G,\eta ,X)$. If $\left(
\pi _{\Phi },v^{\Phi },w^{\Phi },H_{\Phi },K_{\Phi },V_{\Phi },W_{\Phi
}\right) $ is the covariant Stinespring construction associated to $\Phi $,
then $\left( \pi _{\Phi }\times v^{\Phi },H_{\Phi },K_{\Phi },V_{\Phi
},W_{\Phi }\right) $ is the Stinespring construction associated to $\widehat{%
\Phi }$. Indeed, we have 
\begin{eqnarray*}
&&\left[ \{\left( \pi _{\Phi }\times v^{\Phi }\right) \left( f_{x}\right)
V_{\Phi }h;x\in X,f\in C(G,A),h\in H\}\right]  \\
&=&\left[ \pi _{\Phi }\left( X\right) \pi _{\varphi }(C(G,A))V_{\Phi }H%
\right] =\left[ \pi _{\Phi }\left( x\right) H\right] =K
\end{eqnarray*}%
and 
\begin{eqnarray*}
&&\left[ \{\left( \pi _{\Phi }\times v^{\Phi }\right) \left( f_{x}\right)
^{\ast }W_{\Phi }^{\ast }k;x\in X,f\in C(G,A),k\in K\}\right]  \\
&=&\left[ \pi _{\varphi }(C(G,A))^{\ast }\pi _{\Phi }\left( X\right) ^{\ast
}W_{\Phi }^{\ast }K\right] =\left[ \pi _{\varphi }(C(G,A))H\right] =H.
\end{eqnarray*}%
From these relations and taking into account that the map $\widehat{\Phi }$
is defined by $\widehat{\Phi }\left( z\right) =W_{\Phi }^{\ast }\left( \pi
_{\Phi }\times v^{\Phi }\right) \left( z\right) V_{\Phi }$, we deduce that $%
\left( \pi _{\Phi }\times v^{\Phi },H_{\Phi },K_{\Phi },V_{\Phi },W_{\Phi
}\right) $ is the Stinespring construction associated to $\widehat{\Phi }$
(see, Theorem 2.2).
\end{remark}

\textbf{Acknowledgement. }The author would like to thank Professor M. Skeide
for his useful comments.

\end{document}